\documentclass[12pt,reqno]{amsart}

\usepackage[text={152mm,230mm},centering]{geometry}
\usepackage{enumerate}

\usepackage[mathscr]{eucal}
\usepackage[all]{xy}
\usepackage{mathrsfs}
\usepackage{xypic}
\usepackage{amsfonts}
\usepackage{amsmath}
\usepackage{amsthm}
\usepackage{amssymb}
\usepackage{latexsym}
\usepackage{tabularx}
\usepackage{graphicx}
\usepackage{wrapfig}
\usepackage{pict2e}
\usepackage{hyperref}

\usepackage{pdfpages}

\allowdisplaybreaks
\newtheorem{theorem}{Theorem}[section]

\newtheorem{proposition}[theorem]{Proposition}
\newtheorem{lemma}[theorem]{Lemma}

\newtheorem{corollary}[theorem]{Corollary}

\theoremstyle{definition}

\newtheorem{definition}[theorem]{Definition}

\newtheorem{example}[theorem]{Example}

\newtheorem*{ack}{Acknowledgements}

\usepackage{type1cm}
\usepackage[marginal]{footmisc}

\begin{document}

\title{DG Singular equivalence and singular locus}

\author{Leilei Liu}
\author{Jieheng Zeng}

\address{School of Science, Zhejiang University of Science and Technology, Hangzhou, 
Zhejiang Province, 310023, P.R. China}
\email{liuleilei@zust.edu.cn}

\address{School of Mathematical Sciences, Peking University, Beijing 100871 P.R. China}
\email{zengjh662@163.com}

\begin{abstract}
For a commutative Gorenstein Noetherian ring $R$, 
we construct an affine scheme $X$ solely from the DG singularity category $S_{dg}(R)$ of
$R$ such that
there is a finite surjective morphism $X \rightarrow \mathrm{Spec}(R /I)$, 
where $\mathrm{Spec}(R /I)$ is the singular locus in $\mathrm{Spec}(R)$. 
As an application, for such two rings
with equivalent DG singularity categories, we prove that singular loci in their affine schemes have the same dimension.
\end{abstract}
  
\maketitle

\section{Introduction}

\subsection{Backgrounds}
In 1986, Buchweitz introduced a quotient category of an associative algebra $A$, which is called singularity category of $A$ now \cite{B1}. 
The singularity category $D_{sg}(A)$ is defined as the Verdier quotient 
$D^{b}(A) / \mathrm{Perf}(A)$, where $\mathrm{Perf}(A)$ 
is the full subcategory consisting of perfect complexes in the derived category $D^b(A)$ of $A$. 
Hence, $D_{sg}(A)$ measures the smoothness of $A$ in the sense that $A$ is 
homologically smooth if and only if $D_{sg}(A)$ is trivial. 
In that note, Buchweitz showed that $D_{sg}(A)$ is equivalent to the stable category 
$\underline{\mathrm{CM}}(A)$ of Cohen-Macaulay $A$-modules  
as triangulated categories when $A$ is Gorenstein.  
Later Orlov rediscovered the singularity category
via the perspectives of algebraic geometry and mathematical physics. 
Also he recover its deep relationship with Homological Mirror Symmetry \cite{DO}. 
Therefore singularity categories attract the interest of many mathematicians and 
there are many significant progress about singularity categories in various fields, 
such as tilting theory (\cite{IT0, MU0, BIY0} etc.), homological algebra
(\cite{B1, CS0, Ke0, W0, W1}), algebraic geometry
(\cite{MW0, MY0, B0} etc.)
and even in knot theory (\cite{KR}).

Singularity categories are triangulated categories. We know invariants under triangulated equivalence 
play an important role in the research of triangulated categories. 
For example, the Hochschild (co)homology of an associative algebra
are both invariant under the derived Morita equivalence. 
Moreover, Armenta and Keller \cite{AK} showed that the differential calculus
(consisting of Hochschild homology and cohomology endowed with several actions) of an
associative algebra is invariant under the derived Morita equivalence.  
For singularity categories, Wang showed that the Gerstenhaber bracket 
structure on the Tate-Hochschild cohomology is invariant under 
the singular equivalence of Morita type with level (see \cite{W1}). 

In noncommutative algebraic geometry, invariants from geometry are also important and interesting.  
Bondal and Orlov \cite{BO0} showed that for a projective variety $X$ with ample or anti-ample 
canonical bundle, its bounded derived category of coherent sheaves $D^{b}(X)$ 
recovers $X$. 
Recently, Hua and Keller \cite{HK0} showed the coordinate ring of a hypersurface with issolated singularity
can be constructred via it singularity category. The authors prove the result by using an isomorphism between the zeroth 
Tate-Hochschild cohomology and the Tyurina algebra of the hypersurface.

\subsection{Main result}
In \cite{DO1}, Orlov proved that the completion of a variety along 
its singular locus determines its singularity category, 
up to the idempotent completion of a triangulated category.  
Hence there is a natural question: is it true that the singularity category
 of a commutative ring determines its singular locus?
In this paper, We give an confirm answer of this question up to some extent. 

For a commutative Gorenstein Noetherian ring $R$, 
$D_{sg}(R)$ is not enough to detect the geometry generally. 
The result in \cite{D0} motivates us to consider some DG enhancement of $D_{sg}(R)$. 
The singularity category $D_{sg}(R)$ admits a canonical 
DG enhancement given by the DG quotient
$S_{dg}(R):= \mathcal{D}^{b}(R) / \mathcal{P}\mathrm{erf}(R)$ (\cite[\S3]{D0}), 
where $\mathcal{D}^{b}(R)$ is the {\it canonical} DG enhancement 
of $D^{b}(R)$, which then induces a DG enhancement $\mathcal{P}\mathrm{erf}(R)$ of $\mathrm{Perf}(R)$.
Our main result is the following. 

\begin{theorem}\label{Th1}
Let $R$ be a commutative Gorenstein Noetherian ring. Let 
$S_{dg}(R)$ be the DG category described as above.
Then there is an affine scheme $X$ constructed solely from 
$S_{dg}(R)$, and  
a finite surjective morphism $X \rightarrow \mathrm{Spec}(R/I)$, 
where $\mathrm{Spec}(R /I)$ is the singular locus in $\mathrm{Spec}(R)$. 
\end{theorem}

In the above theorem, the singular locus is given as follows.
Let $X$ be a scheme over $k$, then the {\it singular locus} of 
$X$ is  
$$
\mathrm{Sing}(X) := \big\{\mathbf{p} \in X | X_\mathbf{p} \,\textup{is singular}\big\},   
$$
where $X_\mathbf{p}$ is the localization of $X$ at  point $\mathbf{p}$. Note that the scheme structure of $\mathrm{Sing}(X)$ is induced from $X$. When $X$ is an affine scheme, i.e. $X \cong \mathrm{Spec}(R)$
for some commutative Noetherian ring $R$, 
then 
$$
\mathrm{Sing}(X) \cong \mathrm{Spec}(R/I)
$$
as scheme, where $I = \bigcap\limits_{\substack{\mathbf{p} \subseteq R \\R_{\mathbf{p}} \, is \, not \, regular }} \mathbf{p} $. 
Note that $I \cong \sqrt{I}$ as ideal of $R$ and then $\mathrm{Sing}(X)$ is reduced in $\mathrm{Spec}(R)$.  

\subsection{Idea of the proof}\label{IOP}
To prove the theorem, there are $3$ steps as follows
\begin{enumerate}
\item[(1)] First, for any commutative Gorenstein Noetherian ring,    
we show that the coordinate ring of its singular locus is a subring of its zeroth 
Tate-Hochschild cohomology. 

\item[(2)] Second,
 we continue to show that the reduced zeroth Tate-Hochschild cohomology is a finitely 
 generated module over the coordinate ring of its singular locus.

\item[(3)] At last, 
we construct a scheme $X$ with a morphism $X \rightarrow \mathrm{Spec}(R/I)$, 
and prove the Theorem \ref{Th1}.   
\end{enumerate}

The paper is organized as follows. In Section \ref{SS}, we introduce the notions and 
some properties of DG categories, DG singularity categories and DG singular equivalences. 
In Section \ref{TD}, we prove the statements in the steps (1) and (2) above. 
In Section \ref{TR}, we give the proof of the statement in step (3), and get the main result.   

\subsection{Notation and conventions}

Throughout this paper, $k$ is a field with characteristic zero.
In addition, commutative Noetherian rings are assumed to be
finitely generated over $k$ with finite Krull dimensions. 
Unless specified, an associative algebra $A$ is a unital over $k$.
We denote by $\mathrm{mod}(A)$ the category of 
finitely generated right $A$-modules and by $\mathrm{proj}(A)$ the full subcategory 
of $\mathrm{mod}(A)$ consisting of finitely generated projective $A$-modules. Recall 
that $D^b(A)$ is the bounded derived category of $\mathrm{mod}(A)$, 
and $\mathrm{Perf}(A)$ is the bounded homotopy category of complexes of $\mathrm{proj}(A)$.

\section{preliminaries}\label{SS}
\subsection{Singularity category and singular equivalence}\label{ScSe}

In this section, we recall the definitions of singularity category and singular equivalence for algebras. 

\begin{definition}
Let $A$ be an associative algebra. The {\it singularity category} of $A$ 
is defined to be the Verdier quotient $D_{sg}(A) := D^{b}(A)/\mathrm{Perf}(A)$.   
More precisely,
\begin{enumerate}
	\item[(i)] $\mathrm{Ob}\big(D_{sg}(A) \big) = \mathrm{Ob}\big(D^b(A) \big)$; 
	\item[(ii)] for $M^{\ast}, N^{\ast} \in \mathrm{Ob}\big(D_{sg}(A) \big)$, the set
     $\mathrm{Hom}_{D_{sg}(A)}(M^*,N^*)$ is the localization of 
     $\mathrm{Hom}_{D^{b}(A)}(M^*,N^*)$ over $\mathcal{S}$, where $\mathcal{S}$ 
     consists of $f \in \mathrm{Hom}_{D^{b}(A)}(M^{\ast}, N^{\ast})$ 
	satisfying $\mathrm{Cone}(f) \in \mathrm{Ob}\big( \mathrm{Perf}(A) \big)$.
 \end{enumerate}
\end{definition}
Note that $D_{sg}(A)$ is a triangulated category \cite{B1}.
\begin{definition}
Let $A$ and $B$ be associative algebras, then $A$ is said to be singular equivalent to $B$ 
if there exists a triangulated equivalence
$$
D_{sg}(A) \xrightarrow{\sim} D_{sg}(B).  
$$	  
\end{definition}

For example, Chen and Sun introduced the singular 
equivalence of Morita type in \cite{CS0}. More generally,  
Wang introduced singular equivalence of Morita type with level in \cite{W2}. 

\begin{example}[Knorrer's periodicity theorem]\label{EX1}
Let $S : = k[x_1, \cdots, x_m]$ for some integer $m$  and nontrivial element $f \in (x_1, \cdots, x_m)^2$.  
Considering these two algebras $S[u]/(f)$ and $S[u, v]/(f +uv)$, there exist 
an injective morphism
$$ 
S/(f) \hookrightarrow S[u]/(f)
$$
and a surjective morphism
$$
S[u, v]/(f +uv) \twoheadrightarrow S[u]/(f)
$$ 
mapping $v$ to zero. 
Notice that $S[u]/(f)$ is an $\big(S[u, v]/(f +uv)\big) \otimes (S/(f))^{op}$-module.

$S[u]/(f)$, viewed as an $S[u, v]/(f +uv)$-module, 
admits a projective resolution 
$$
0 \rightarrow S[u, v]/(f +uv) \xrightarrow{v \cdot} S[u, v]/(f +uv) \rightarrow  S[u]/(f) \rightarrow 0.
$$
It follows that $S[u]/(f) \in \mathrm{Perf}\big( S[u, v]/(f +uv) \big)$.  
Thus, there is a triangle functor 
$$
D_{sg}\big(S/(f)\big) \xrightarrow{(-)\otimes^{\mathbb{L}}_{S/(f)} 
S[u]/(f)} D_{sg}\big( S[u, v]/(f +uv) \big). 
$$
In \cite{K1}, Knorrer 
proved that the above triangle functor is a singular 
equivalence between $S/(f)$ and $S[u, v]/(f +uv)$. 
\end{example}


\subsection{DG category and DG functor}
\begin{definition}
A {\it DG category} is a $k$-linear category $\mathcal{D}$ 
such that the Hom-set $\mathrm{Hom}_{\mathcal{D}}(X, Y)$ 
consists of complexes of vector spaces with the associative composition  
$$
\mathrm{Hom}_{\mathcal{D}}(Y, Z) \otimes \mathrm{Hom}_{\mathcal{D}}(X, Y) \rightarrow \mathrm{Hom}_{\mathcal{D}}(X, Z),
$$ 
and the identity morphisms are 
closed in degree zero.
\end{definition}

A {\it DG functor} $F: \mathcal{D} \rightarrow \mathcal{D}'$ between DG categories is a functor such that the map 
$$
F: \mathrm{Hom}_{\mathcal{D}}(X, Y) \rightarrow \mathrm{Hom}_{\mathcal{D}'}\big(F(X), F(Y)\big)
$$
is a complex morphism.
{definition}

For example, a DG algebra is a DG category with one object. 
A homomorphism between DG algebras can be viewed as 
a DG functor between the corresponding DG categories.  

For any DG category $\mathcal{D}$, there is a $k$-linear graded  
category $\mathrm{H}(\mathcal{D})$ given as follows: 
\begin{enumerate}
	\item[(i)] $\mathrm{Ob}\big(\mathrm{H}(\mathcal{D})\big) = \mathrm{Ob}(\mathcal{D})$; 
	\item[(ii)] $\mathrm{Hom}_{\mathrm{H}(\mathcal{D})}(X, Y) 
	= \mathrm{H}^{\ast}\big( \mathrm{Hom}_{\mathcal{D}}(X, Y) \big) 
	:= \bigoplus\limits_{i\in\mathbb{Z}} \mathrm{H}^{i}
	\big( \mathrm{Hom}_{\mathcal{D}}(X, Y) \big)$, for any objects $X, Y$.
\end{enumerate}
Furthermore, denote by $\mathrm{H}^{0}(\mathcal{D})$ the category:
\begin{enumerate}
	\item[(i)] $\mathrm{Ob}\big(\mathrm{H}^{0}(\mathcal{D})\big) 
	= \mathrm{Ob}\big(\mathcal{D}\big)$; 
	\item[(ii)] $\mathrm{Hom}_{\mathrm{H}^{0}(\mathcal{D})}(X, Y) 
	= \mathrm{H}^{0}\big( \mathrm{Hom}_{\mathcal{D}}(X, Y) \big)$. 
\end{enumerate}
Naturally, any DG functor $F$ induces two functors $\mathrm{H}(F)$ and $\mathrm{H}^{0}(F)$. 

\begin{definition}
A DG functor $F: \mathcal{D} \rightarrow \mathcal{D}'$ is called a {\it quasi-equivalence} 
if 
$$
F: \mathrm{Hom}_{\mathcal{D}}(X, Y) \rightarrow 
\mathrm{Hom}_{\mathcal{D}'}\big(F(X), F(Y)\big)
$$
is a quasi-isomorphism, for any $X, Y \in \mathrm{Ob}(\mathcal{D})$, 
and $\mathrm{H}^{0}(F)$ is an equivalence. 	
\end{definition}

Now, let $\mathrm{DGCat}$ be the category whose objects are small 
DG categories and whose morphisms are DG functors. 
Consider the localization, denoted by $\mathrm{Hqe}$, of  $\mathrm{DGCat}$ 
with respect to quasi-equivalences. 
We call any morphism in $\mathrm{Hqe}$ a quasi-functor. 

For more properties of DG categories such as DG modules of DG categories 
and tensor functors between them, we refer to the papers \cite{Ke1} and \cite{KL0}.   
 
Given a small DG category $\mathcal{C}$, 
let $\mathrm{DGMod}(\mathcal{C})$ be the DG category of $\mathcal{C}$-modules, which is defined to be the set   
$\mathrm{Hom}(\mathcal{C}^{op} , \mathcal{D}(k))$ of DG functors,  
where $\mathcal{D}(k)$ is the canonical DG category 
of complexes of $k$-linear vector spaces. 
A $\mathcal{C}$-module is {\it representable} if it is contained in the essential image 
of the Yoneda DG functor 
$$
\mathrm{Y}_{dg}^{\mathcal{C}}: \mathcal{C} \rightarrow \mathrm{DGMod}(\mathcal{C}) ,
\,\,\,  X \mapsto \mathrm{Hom}_{\mathcal{C}}(-, X).  
$$

\begin{definition}
A DG category $\mathcal{C}$ is called {\it pretriangulated} if the essential image of the functor 
$$
\mathrm{H}^{0}\big(\mathrm{Y}_{dg}^{\mathcal{C}}\big): 
\mathrm{H}^{0}(\mathcal{C}) \rightarrow \mathrm{H}^{0}\big(\mathrm{DGMod}(\mathcal{C}) \big)
$$ 
is a triangulated subcategory.  	
\end{definition}




\section{The Tate-Hochschild cohomology}\label{TD}

\subsection{Generalized Tate-Hochschild complex of associative algebra}

The notion of Tate-Hochschild cohomology 
was introduced by Buchweitz  (see \cite{B1}).

\begin{definition}
Let $\Lambda$ be an associative algebra. Its $i$-th 
{\it Tate-Hochschild cohomology}, denoted by $\mathrm{HH}_{sg}^{i}(\Lambda)$, is  
$
\mathrm{Hom}_{D_{sg}(\Lambda^e)} (\Lambda, \Lambda[i]),
$
where $\Lambda^e := \Lambda \otimes \Lambda^{op}$. 	
\end{definition}

Later, for any associative algebra $\Lambda$, 
Wang defined the singular Hochschild  
complex whose cohomologies are the Tate-Hochschild cohomologies (\cite{W0}). 
This complex is constructed as the colimit of a sequence of complexes. 
He also constructed a complex, called the {\it generalized 
Tate-Hochschild complex} of $\Lambda$: 
$$
\mathcal{D}^{\ast}(\Lambda, \Lambda): \cdots \xrightarrow{b_2} C_{1}(\Lambda, \Lambda^{\vee}) \xrightarrow{b_1}  \Lambda^{\vee} \xrightarrow{\mu} \Lambda \xrightarrow{\delta^0} C^{1}(\Lambda, \Lambda) \xrightarrow{\delta^1} \cdots   
$$
whose cohomologies, in the case of $\Lambda$ being self-injective, are exactly the 
Tate-Hochschild cohomologies, where 
$\Lambda^{\vee} := \mathrm{Hom}_{\Lambda^e}(\Lambda, \Lambda \otimes \Lambda)$, 
 the differential $\mu$ is given by the multiplication 
 of $\Lambda$: $\mu(x \otimes y)=xy$, $C_{\ast}(\Lambda, \Lambda^{\vee})$ 
 is the Hochschild chain complex 
of $\Lambda$ with $\Lambda^{e}$-module $\Lambda^{\vee}$, and
$C^{\ast}(\Lambda, \Lambda)$ 
is the Hochschild cochain complex of $\Lambda$. 

In the following, we introduce a new  
complex for an
arbitrary associative algebra $A$ such that its 
cohomologies are also the Tate-Hochschild cohomologies (see
Proposition \ref{MainPro} below). 
The underlying vector space of this complex is described explicitly. 
Moreover, when $A$ is self-injective, this complex coincides with
 $\mathcal{D}^{\ast}(A, A)$. 

Let us first recall the DG quotient of a DG category, introduced by Drinfeld (see \cite{D0}).  
 
\begin{definition}
Let $\mathcal{A}$ be a DG category and $\mathcal{B} \subset \mathcal{A}$ 
be a full DG subcategory of $\mathcal{A}$. 	
A {\it DG quotient} of $\mathcal{A}$ by $\mathcal{B}$ is a diagram 
consisting of DG categories and DG functors
$$
\mathcal{A} \xleftarrow{\sim} \tilde{\mathcal{A}} \xrightarrow{\pi} \mathcal{C}
$$
satisfying that 
\begin{enumerate}
\item[(1)] the above DG functor $\tilde{\mathcal{A}} 
\xrightarrow{\sim} \mathcal{A}$ is a quasi-equivalence;
  
\item[(2)] the functor $\mathrm{H}(\pi): 
\mathrm{H}(\tilde{\mathcal{A}}) \rightarrow \mathrm{H}(\mathcal{C})$ 
between corresponding homotopy categories induced by $\pi$ is essentially surjective; 
 	
\item[(3)] $\mathrm{H}(\pi)$ gives a triangle functor 
$\tilde{\mathcal{A}}^{tr} \rightarrow \mathcal{C}^{tr}$, which induces an equivalence:     
$$
\mathcal{A}^{tr}/\mathcal{B}^{tr}  \xrightarrow{\sim} \mathcal{C}^{tr}, 
$$ 
where $\mathcal{A}^{tr}/\mathcal{B}^{tr}$ is the Verdier quotient of 
$\mathcal{A}^{tr}$ by $\mathcal{B}^{tr}$, 
 $\mathcal{A}^{tr}$ 
(resp. $\tilde{\mathcal{A}}^{tr}, \mathcal{B}^{tr}, \mathcal{C}^{tr}$) 
represents the triangulated category $\mathrm{H}(\mathcal{A}^{pretr})$ (resp. $\mathrm{H}(\tilde{\mathcal{A}}^{pretr}), 
\mathrm{H}(\mathcal{B}^{pretr}), \mathrm{H}(\mathcal{C}^{pretr})$), and $\mathcal{A}^{pretr}$ (resp. $\tilde{\mathcal{A}}^{pretr}, \mathcal{B}^{pretr}, \mathcal{C}^{pretr}$)
is a certain pre-triangulated DG category on $\mathcal{A}$ (resp. $\tilde{\mathcal{A}}, \mathcal{B}, \mathcal{C}$). 
\end{enumerate}	
\end{definition}

From \cite[\S3]{D0}, we know that, for any DG category $\mathcal{A}$ 
and its full DG subcategory $\mathcal{B} \subset \mathcal{A}$, the DG quotient of $\mathcal{A}$ by 
$\mathcal{B}$ exists. On the other hand, 
it is well-known that both $D^{b}(A)$ and $\mathrm{Perf}(A)$ are DG categories. 
Here, we denote by $\mathcal{P}\mathrm{erf}(A)$ and $\mathcal{D}^{b}(A)$ the DG 
categories corresponding to $\mathrm{Perf}(A)$ and $D^{b}(A)$ respectively. Moreover, both  
 $\mathcal{P}\mathrm{erf}(A)$ and $\mathcal{D}^{b}(A)$ have canonical pre-triangulated structures.  
Hence, 
by taking the DG quotient, $D_{sg}(A)$ is 
endowed with a DG category structure which is a DG enhancement for $D_{sg}(A)$. 
We denote by $S_{dg}(A)$ this DG category associated to $D_{sg}(A)$.

\begin{definition}
Let $A$ and $B$ be two associative algebras. $A$ and $B$ are called {\it DG singular 
equivalent} if there is a quasi-equivalence  
$$
S_{dg}(A) \xrightarrow{\sim} S_{dg}(B).  
$$	  	
\end{definition}

It is well-known that the DG singular 
equivalence between $S_{dg}(A)$ and $S_{dg}(B)$ implies that $D_{sg}(A) \simeq D_{sg}(B)$.
But, in general, singular equivalence cannot be lifted to DG singular equivalence. 
However, we have the following proposition for some special cases.

\begin{proposition}[{\cite[Proposition 3.1]{CLR}}]\label{Morita}
Keep the settings as above. Let $M$ be a $B \otimes A^{op}$-module which is a
projective $B$-module 
and projective $A^{op}$-module. 
Then the following statements are equivalent.
\begin{enumerate}
	\item[(1)] The triangle functor
	$$
	(-) \otimes^{\mathbb{L}}_{A} M: D_{sg}(A) \rightarrow D_{sg}(B)
	$$ 
	is an equivalence; 
	\item[(2)] The quasi-functor
	$$
	(-) \otimes^{\mathbb{L}}_{A} M: 
S_{dg}(A) \rightarrow S_{dg}(B)	
	$$ is a quasi-equivalence. 
\end{enumerate}
\end{proposition}
Here, we call singularity equivalence given by some 
$B \otimes A^{op}$-module $M$ 
as above proposition {\it singularity equivalence of Morita type} (\cite{CS0}).

\begin{example}	
Let us back to Example \ref{EX1}. 
In the argument of Example \ref{EX1}, 
we know that $M$ is a projective $S[u, v]/ (f +uv)$-module, 
where $M := S[u]/(f)$. Moreover, 
it is easy to check that $M$ is a projective  $S^{op}$-module.
Hence, we get a triangle functor 
$$
(-) \otimes^{\mathbb{L}}_{S/(f)} M: D_{sg}(S/(f)) \rightarrow D_{sg}\big(S[u, v]/(f + uv)\big). 
$$
By Knorrer's periodicity theorem, this 
functor is a triangle equivalence. Finally, by Proposition \ref{Morita}, 
we get that it is in fact an equivalence
$$
S_{dg}(S/(f)) \simeq S_{dg}\big(S[u, v]/(f + uv) \big). 
$$ 
  
\end{example}

In \cite[Theorem 1.1]{Ke0}, Keller realized the Tate-Hochschild cohomologies 
of an algebra $A$ as the Hochschild cohomologies of $S_{dg}(A)$.  

\begin{theorem}[\cite{Ke0},Theorem 1.1]\label{Th2}
 There is a canonical isomorphism of graded algebras between the 
 Tate-Hochschild cohomologies $\mathrm{HH}^{\ast}_{sg}(A)$ 
 of $A$ and 
 the Hochschild cohomologies of $S_{dg}(A)$. 
\end{theorem}

 Next we recall the following proposition (see \cite[1.3.1]{D0}). 
  
\begin{proposition}\label{DGHo0}
Let $M^{\ast}, N^{\ast} \in \mathrm{Ob}\big(S_{dg}(A)\big)$. Then 
$$
\mathrm{Hom}_{S_{dg}(A)}(M^{\ast}, N^{\ast}) \cong 
\mathrm{Cone}\big(h_{N^{\ast}} \otimes_{\mathcal{P}\mathrm{erf}(A)}^{\mathbb{L}} \tilde{h}_{M^{\ast}}  
\rightarrow  \mathrm{Hom}_{\mathcal{D}^{b}(A)}(M^{\ast}, N^{\ast})\big)\quad\textup{in}\,\, D(k),$$
where $h_{N^{\ast}}(-) := \mathrm{Hom}_{\mathcal{D}^{b}(A)}(-, N^{\ast})$ 
as a $\mathcal{P}\mathrm{erf}(A)$-module, $\tilde{h}_{M^{\ast}}(-) := \mathrm{Hom}_{\mathcal{D}^{b}(A)}(M^{\ast}, -)$  
as a $\mathcal{P}\mathrm{erf}(A)^{op}$-module 
and $\mathcal{P}\mathrm{erf}(A)^{op}$ is the opposite DG category of $\mathcal{P}\mathrm{erf}(A)$. 	
\end{proposition}

In the above proposition,
the morphism $h_{N^{\ast}} \otimes_{\mathcal{P}\mathrm{erf}(A)}^{\mathbb{L}} \tilde{h}_{M^{\ast}} 
 \rightarrow  \mathrm{Hom}_{\mathcal{D}^{b}(A)}(M^{\ast}, N^{\ast})$ is given by the composition of morphisms in $\mathcal{D}^{b}(A)$. 
In the meantime, we have the following. 

\begin{lemma}\label{DGHo1}
With the setting of the above proposition, 
$$
h_{N^{\ast}} \otimes_{\mathcal{P}\mathrm{erf}(A)}^{\mathbb{L}} \tilde{h}_{M^{\ast}} 
\cong N^{\ast} \otimes^{\mathbb{L}}_{A} (M^{\ast})^{\vee^ \mathbb{L}}  
$$
in $D(k)$, where $(M^{\ast})^{\vee^ \mathbb{L}}
: = \mathbb{R}\mathrm{Hom}_{A}(M^{\ast}, A)$ is an object of $D^{b}(A^{op})$. 	
\end{lemma}

\begin{proof}
Consider the natural inclusion functor of DG categories  
$$
 \Psi_A: A^{op} \rightarrow \mathcal{P}\mathrm{erf}(A),
$$
sending the unique object of $A^{op}$ to $A$-module complex $A$. 
By the result of Keller in \cite[Theorem 8.1]{Ke1}, the induced functor 
$$
(\Psi_A)_{*}: D\big( \mathcal{P}\mathrm{erf}(A) \big) \rightarrow D(A^{op})
$$ gives a triangulated equivalence. 
In the same way, we get the equivalence
$$
(\Psi^{op}_A)_{*}: D\big( \mathcal{P}\mathrm{erf}(A)^{op} \big) \rightarrow D(A). 
$$ 

By the definition of tensor product of DG modules over 
a DG category (see \cite[3.5]{KL0}), we know that  
\begin{align*}
& h_{N^{\ast}} \otimes_{\mathcal{P}\mathrm{erf}(A)} \tilde{h}_{M^{\ast}} \\
& \cong \mathrm{Cone}\Big( \bigoplus_{X, Y \in \mathcal{P}\mathrm{erf}(A)} h_{N^{\ast}}(X) \otimes_{k} 
 \mathrm{Hom}_{\mathcal{P}\mathrm{erf}(A)}(Y, X)  \otimes_{k} \tilde{h}_{M^{\ast}}(Y) \\
&\quad\quad\quad\quad\quad\quad \rightarrow \bigoplus_{X \in \mathcal{P}\mathrm{erf}(A)} h_{N^{\ast}}(X) \otimes_{k} \tilde{h}_{M^{\ast}}(X) \Big)\\
& \cong \mathrm{Cone}\Big( \bigoplus_{X, Y \in \mathcal{P}\mathrm{erf}(A)} \mathrm{Hom}_{\mathcal{D}^{b}(A)}(X, N^{\ast}) \otimes_{k}  
\mathrm{Hom}_{\mathcal{P}\mathrm{erf}(A)}(Y, X)  \otimes_{k} \mathrm{Hom}_{\mathcal{D}^{b}(A)}(M^{\ast}, Y) \\
&\quad\quad\quad\quad\quad\quad
\rightarrow \bigoplus_{X \in \mathcal{P}\mathrm{erf}(A)} \mathrm{Hom}_{\mathcal{D}^{b}(A)}(X, N^{\ast})
 \otimes_{k} \mathrm{Hom}_{\mathcal{D}^{b}(A)}(M^{\ast}, X) \Big). 
\end{align*}
If we replace $N^{\ast}$
by its projective $A$-module resolution $P_{N^{*}}$, $h_{N^{\ast}} \otimes_{\mathcal{P}\mathrm{erf}(A)}^{\mathbb{L}} \tilde{h}_{M^{\ast}}$ is isomorphic to
$
h_{P_{N^{\ast}}} \otimes_{\mathcal{P}\mathrm{erf}(A)} \tilde{h}_{M^{\ast}}$ by the result of
\cite[3.5]{KL0}.
It follows that 
\begin{align*}
& h_{N^{\ast}} \otimes_{\mathcal{P}\mathrm{erf}(A)}^{\mathbb{L}} \tilde{h}_{M^{\ast}} \\
& \cong \mathrm{Cone}\Big( \bigoplus_{X, Y \in \mathcal{P}\mathrm{erf}(A)} 
h_{P_{N^{\ast}}}(X) \otimes_{k}  \mathrm{Hom}_{\mathcal{P}\mathrm{erf}(A)}(Y, X)   \otimes_{k} \tilde{h}_{M^{\ast}}(Y) \\
&\quad\quad\quad\quad\quad\quad
 \rightarrow \bigoplus_{X \in \mathcal{P}\mathrm{erf}(A)} h_{P_{N^{\ast}}}(X) \otimes_{k} \tilde{h}_{M^{\ast}}(X) \Big)\\
& \cong \mathrm{Cone}\Big( \bigoplus_{X, Y \in \mathcal{P}\mathrm{erf}(A)} 
\mathrm{Hom}_{\mathcal{D}^{b}(A)}(X, P_{N^{\ast}}) \otimes_{k}  \mathrm{Hom}_{\mathcal{P}\mathrm{erf}(A)}(Y, X)  
\otimes_{k} \mathrm{Hom}_{\mathcal{D}^{b}(A)}(M^{\ast}, Y) \\
& \quad\quad\quad\quad\quad\quad\rightarrow
 \bigoplus_{X \in \mathcal{P}\mathrm{erf}(A)} \mathrm{Hom}_{\mathcal{D}^{b}(A)}
 (X, P_{N^{\ast}}) \otimes_{k} \mathrm{Hom}_{\mathcal{D}^{b}(A)}(M^{\ast}, X) \Big). 
\end{align*}
On the other hand, we have 
\begin{align*}
  N^{\ast} \otimes^{\mathbb{L}}_{A} (M^{\ast})^{\vee^ \mathbb{L}} & \cong  P_{N^{\ast}} \otimes_{A} (M^{\ast})^{\vee^ \mathbb{L}} \\
  & \cong \mathrm{Hom}_{\mathcal{D}^{b}(A)}(A, P_{N^{\ast}}) \otimes_{A} \mathrm{Hom}_{\mathcal{D}^{b}(A)}(M^{\ast}, A)\\
  & \cong  \mathrm{Cone}\Big( \mathrm{Hom}_{\mathcal{D}^{b}(A)}(A, P_{N^{\ast}})
   \otimes \mathrm{Hom}_{\mathcal{D}^{b}(A)}(A, A) \otimes \mathrm{Hom}_{\mathcal{D}^{b}(A)}(M^{\ast}, A) \\
  &\quad\quad\quad\quad\quad\quad
   \rightarrow  \mathrm{Hom}_{\mathcal{D}^{b}(A)}(A, P_{N^{\ast}}) \otimes \mathrm{Hom}_{\mathcal{D}^{b}(A)}(M^{\ast}, A) \Big).
\end{align*} 
Hence, we obtain a morphism between complexes  
$$
\eta_{N^{*}}^{M^*}: N^{\ast} \otimes^{\mathbb{L}}_{A} (M^{\ast})^{\vee^ \mathbb{L}} \hookrightarrow h_{N^{\ast}} \otimes_{\mathcal{P}\mathrm{erf}(A)}^{\mathbb{L}} \tilde{h}_{M^{\ast}}
$$
given by 
\begin{align*}
& \mathrm{Cone}\Big( \mathrm{Hom}_{D^{b}(A)}(A, P_{N^{\ast}}) \otimes 
\mathrm{Hom}_{D^{b}(A)}(A, A) \otimes \mathrm{Hom}_{D^{b}(A)}(M^{\ast}, A) \\
  &\quad\quad\quad\quad\quad\quad
   \rightarrow  \mathrm{Hom}_{D^{b}(A)}(A, P_{N^{\ast}}) \otimes 
   \mathrm{Hom}_{D^{b}(A)}(M^{\ast}, A) \Big)  \\
\hookrightarrow& \mathrm{Cone}\Big( \bigoplus_{X, Y \in \mathcal{P}\mathrm{erf}(A)} 
\mathrm{Hom}_{\mathcal{D}^{b}(A)}(X, P_{N^{\ast}}) \otimes_{k}  \mathrm{Hom}_{\mathcal{P}\mathrm{erf}(A)}(Y, X)  \otimes_{k} \mathrm{Hom}_{\mathcal{D}^{b}(A)}(M^{\ast}, Y) \\
&\quad\quad\quad\quad\quad\quad
 \rightarrow \bigoplus_{X \in \mathcal{P}\mathrm{erf}(A)} 
 \mathrm{Hom}_{\mathcal{D}^{b}(A)}(X, P_{N^{\ast}}) \otimes_{k} \mathrm{Hom}_{\mathcal{D}^{b}(A)}(M^{\ast}, X) \Big). 
\end{align*}
Obviously, $\eta_{N^{*}}^{M^*}$ is a quasi-isomorphism if and only if the induced morphism 
$$
\mathbb{R}\mathrm{Hom}_{k}(\eta_{N^{*}}^{M^*}, k): 
\mathbb{R}\mathrm{Hom}_{k}\big(h_{N^{\ast}} \otimes_{\mathcal{P}\mathrm{erf}(A)}^{\mathbb{L}} \tilde{h}_{M^{\ast}}, k\big) 
\twoheadrightarrow
\mathbb{R}\mathrm{Hom}_{k}\big(N^{\ast} \otimes^{\mathbb{L}}_{A} (M^{\ast})^{\vee^ \mathbb{L}}, k\big)
$$
is a quasi-isomorphism. 

Now, by the property of adjunction functors, we obtain 
$$
\mathbb{R}\mathrm{Hom}_{k}\big(h_{N^{\ast}} \otimes_{\mathcal{P}\mathrm{erf}(A)^{op}}^{\mathbb{L}} \tilde{h}_{M^{\ast}}, k\big) \cong \mathrm{Hom}_{\mathcal{P}\mathrm{erf}(A)}\big(h_{N^{\ast}},  \mathbb{D}(\tilde{h}_{M^{\ast}}) \big)
$$
and 
$$
\mathbb{R}\mathrm{Hom}_{k}\big(N^{\ast} \otimes^{\mathbb{L}}_{A} (M^{\ast})^{\vee^ \mathbb{L}}, k\big) \cong \mathbb{R}\mathrm{Hom}_{A}\big(N^{*}, \mathbb{D}((M^{\ast})^{\vee^ \mathbb{L}}) \big),
$$
where $\mathbb{D}(-) := \mathbb{R}\mathrm{Hom}_{k}(-, k)$. 
Note that for any object $X$ in $\mathcal{P}\mathrm{erf}(A)$, 
$\mathbb{D}(\tilde{h}_{M^{\ast}})(X) := \mathbb{D}\big(\tilde{h}_{M^{\ast}}(X)\big)$. 
Meanwhile, we have
$$
(\Psi^{op}_A)_{*}(h_{N^*}) \cong N^{*}\quad
\textup{
and}
\quad
(\Psi^{op}_A)_{*}(\tilde{h}_{M^*}) \cong (M^{\ast})^{\vee^ \mathbb{L}}.
$$
Hence,
$$
(\Psi^{op}_A)_{*}\big(\mathbb{D}(\tilde{h}_{M^*})\big) \cong \mathbb{D}\big((M^{\ast})^{\vee^ \mathbb{L}}\big).
$$
Since $\Psi^{op}_A$ gives a triangulated equivalence, we obtain that
$\mathrm{Hom}_{\mathcal{P}\mathrm{erf}(A)^{op}}\big(h_{N^{\ast}}, 
 \mathbb{D}(\tilde{h}_{M^{\ast}}) \big)$ is isomorphic to 
$\mathbb{R}\mathrm{Hom}_{A}\big(N^{*}, \mathbb{D}((M^{\ast})^{\vee^ \mathbb{L}}) \big)$ 
in $D(k)$. 
Then $\mathbb{R}\mathrm{Hom}_{k}(\eta_{N^{*}}^{M^*}, k)$ is a quasi-isomorphism, which implies
\begin{align*}
h_{N^{\ast}} \otimes_{\mathcal{P}\mathrm{erf}(A)}^{\mathbb{L}} \tilde{h}_{M^{\ast}} \cong N^{\ast} \otimes^{\mathbb{L}}_{A} (M^{\ast})^{\vee^ \mathbb{L}} 
\end{align*}
in $D(k)$. Thus, we complete the proof. 
\end{proof}

Combining Proposition \ref{DGHo0} with Lemma \ref{DGHo1},  we obtain  
$$
\mathrm{Hom}_{S_{dg}(A)}(M^{\ast}, N^{\ast}) \cong \mathrm{Cone}\big(
N^{\ast} \otimes^{\mathbb{L}}_{A} (M^{\ast})^{\vee^ \mathbb{L}} \rightarrow 
\mathrm{Hom}_{\mathcal{D}^{b}(A)}(M^{\ast}, N^{\ast})\big)  
$$
in $D(k)$. Now let $R$ be a commutative Noetherian ring, and set $A = R^{e}$ 
and $M^{\ast} = N^{\ast} = R$ as $R \otimes R^{op}$-module in the above isomorphism, we get 
\begin{align*}
& \mathrm{HH}^{i}_{sg}(R) := \mathrm{Hom}_{D_{sg}(R^e)}(R, R[i]) = \mathrm{H}^{i}\big(\mathrm{Hom}_{S_{dg}(R^e)}(R, R)\big) \\
& \cong \mathrm{H}^{i}\left(\mathrm{Cone}\big(
R \otimes^{\mathbb{L}}_{R^e} (R)^{\vee^ \mathbb{L}} \rightarrow \mathrm{Hom}_{\mathcal{D}^{b}(R^e)}(R, R)\big)\right) 
\end{align*} 
for any $i \in \mathbb{Z}$.

Now let us consider the Bar resolution $\mathrm{Bar}(R) \twoheadrightarrow R$ over $R^e$. 
There is a double complex: 
$$
\scalebox{0.8}{
\xymatrixcolsep{1pc}
\xymatrix{ 
 & 0  \ar[d] & 0  \ar[d]  \\
 \ar[r] & (R \otimes R \otimes R) \otimes_{R^e}
 \mathrm{Hom}_{R^e}\big( R \otimes R,  R \otimes R) 
  \ar[r] \ar[d] & (R \otimes R) \otimes_{R^e} 
  \mathrm{Hom}_{R^e}\big( R \otimes R,  R \otimes R)   
  \ar[r] \ar[d] & 0  \\
\ar[r] & (R \otimes R \otimes R) \otimes_{R^e} 
\mathrm{Hom}_{R^e}\big( R \otimes R \otimes R,  R \otimes R)   
\ar[r] \ar[d] & (R \otimes R) \otimes_{R^e} 
\mathrm{Hom}_{R^e}\big( R \otimes R \otimes R,  R \otimes R) \ar[r] \ar[d] & 0 \\
  & \vdots  & \vdots   }
}
$$
given by $\mathrm{Bar}(R) \otimes_{R^e} \mathbb{R}\mathrm{Hom}_{R^e}\big(\mathrm{Bar}(R), R^e \big)$. 
View this double complex as a cochain complex, 
and denote it by $E^{\ast}$. 
Also, denote by $F^\ast$ the cochain complex obtained from the double
complex $\mathbb{R}\mathrm{Hom}_{R^e}(\mathrm{Bar}(R), \mathrm{Bar}(R))$. 
We get a complex
$\mathrm{Cone}(E^{\ast} \rightarrow F^{\ast} )$
and we obtain the following proposition.

\begin{proposition}\label{MainPro}
With the above settings, 
$$
\mathrm{HH}^{i}_{sg}(R) \cong \mathrm{H}^{i}\big(\mathrm{Cone}(
E^{\ast} \rightarrow F^{\ast} ) \big)
$$
for any $i \in \mathbb{Z}$. 
\end{proposition}

\subsection{Hochschild cohomology of DG singularity category}

In this subsection, we give proofs of the statements in steps (1) and (2) in \S\ref{IOP}. 

Our goal is to construct a multiplication on $\mathrm{Cone}(E^{\ast} \rightarrow F^* )$. 
First, we know that
the composition of elements in the endomorphism ring
$F^{\ast} = \mathbb{R}\mathrm{Hom}_{R^e}(\mathrm{Bar}(R), \mathrm{Bar}(R))$ is a natural multiplication. 
Second, the following map 
\begin{align*}
E^{\ast} \otimes E^{\ast} \cong & \Big( \mathrm{Bar}(R) \otimes_{R^e} \mathbb{R}\mathbb{R}\mathrm{Hom}_{R^e}\big(\mathrm{Bar}(R), R^e \big) \Big) 
\otimes \Big( \mathrm{Bar}(R) \otimes_{R^e} \mathbb{R}\mathrm{Hom}_{R^e}\big(\mathrm{Bar}(R), R^e \big) \Big) \\
& \cong  \mathrm{Bar}(R) \otimes_{R^e} \Big( \mathbb{R}\mathrm{Hom}_{R^e}\big(\mathrm{Bar}(R), R^e \big)
 \otimes \mathrm{Bar}(R) \Big) \otimes_{R^e} \mathbb{R}\mathrm{Hom}_{R^e}\big(\mathrm{Bar}(R), R^e \big) \\
& \xrightarrow{invol.}   \mathrm{Bar}(R) \otimes_{R^e} 
R^e \otimes_{R^e} \mathbb{R}\mathrm{Hom}_{R^e}\big(\mathrm{Bar}(R), R^e \big) 
\\ 
& \cong \mathrm{Bar}(R) \otimes_{R^e} \mathbb{R}\mathrm{Hom}_{R^e}\big(\mathrm{Bar}(R), R^e \big) = E^{\ast}, 
\end{align*}
induced by involution, gives a  natural multiplication structure on $E^{\ast}$. 
Finally, there is a canonical 
$\Big( \mathbb{R}\mathrm{Hom}_{R^e}\big(\mathrm{Bar}(R), \mathrm{Bar}(R)\big) \Big)^e$-module 
structure
on $\mathrm{Bar}(R) \otimes_{R^e} \mathbb{R}\mathrm{Hom}_{R^e}\big(\mathrm{Bar}(R), R^e \big)$. 
Therefore, viewing
$\mathrm{Cone}(E^{\ast} \rightarrow F^* )$ as a semi-product $E^{\ast} \rtimes F^{\ast}$,
we obtain a natural
multiplication on $\mathrm{Cone}(E^{\ast} \rightarrow F^* )$. 

Next, recall that the canonical morphism $E^{\ast} \rightarrow F^{\ast}$
is given by the following composition map:
\begin{align*}
\mathrm{Bar}(R) \otimes_{R^e} \mathbb{R}\mathrm{Hom}_{R^e}\big(\mathrm{Bar}(R), R^e \big) 
& \cong \mathbb{R}\mathrm{Hom}_{R^e}\big( R^{e}, \mathrm{Bar}(R) \big) \otimes_{R^e} \mathbb{R}\mathrm{Hom}_{R^e}\big(\mathrm{Bar}(R), R^e \big) \\
& \rightarrow \mathbb{R}\mathrm{Hom}_{R^e}(\mathrm{Bar}(R), \mathrm{Bar}(R)).  
\end{align*}
This morphism together with differentials on 
$F^\ast$ and $E^\ast$ gives the natural differential on
$\mathrm{Cone}(E^{\ast} \rightarrow F^* )$. 
It is easy to check that the differential is compatible with 
the multiplication on $\mathrm{Cone}(E^{\ast} \rightarrow F^* )$. 
Thus, $\mathrm{Cone}(E^{\ast} \rightarrow F^* )$ is a differential graded algebra.

In the following, we recall the support scheme of object in $D(R)$ and
the diagonal support scheme of $S_{sg}(R)$. 

\begin{definition}
(1) Let $Q^{\ast}$ be an object in $D(R)$. The {\it support scheme} 
$\mathrm{Supp}(Q^{\ast})$ 
of $Q^{\ast}$ is a subscheme of $\mathrm{Spec}(R)$:
$$
\mathrm{Supp}(Q^{\ast}): = \big\{\mathbf{p} \in\mathrm{Spec}(R) \mid Q^{\ast}_{\mathbf{p}} 
\ncong 0 \, \textup{in} \, D(R_{\mathbf{p}}) \big\} . 
$$

(2) The {\it diagonal support scheme} 
$\mathrm{DSupp}(S_{dg}(R))$ of $S_{dg}(R)$ is a subscheme of $\mathrm{Spec}(R)$: 
$$
\mathrm{DSupp}(S_{dg}(R)): = \big\{\mathbf{p} \in\mathrm{Spec}(R) \mid \mathrm{Hom}_{S_{dg}(R^e)}(R, R)_{\mathbf{p}} 
\ncong 0 \, \textup{in} \, D(R_{\mathbf{p}}) \big\}. 
$$
\end{definition}

We have the following.

\begin{proposition}\label{Injec}
Let $R$ be a commutative Noetherian ring. Let $R/I$ be the coordinate
ring of the singular locus in $\mathrm{Spec}(R)$. Then there is an algebraic injection
$$
\iota_I: R/I \hookrightarrow \mathrm{HH}^{0}_{sg}(R)_{red},  
$$
where $\mathrm{HH}^{0}_{sg}(R)_{red}$ is the quotient ring of $\mathrm{HH}^{0}_{sg}(R)$ by its nilpotent ideal. 
\end{proposition}

\begin{proof}
First, there is a canonical DG algebraic homomorphism
$$
F^{\ast} \hookrightarrow \mathrm{Cone}(E^{\ast} \rightarrow F^* ),
$$	
which induces the following algebraic homomorphism
$$
\pi_R: R \cong \mathrm{HH}^{0}(R) \cong \mathrm{H}^{0}(F^*) \rightarrow  
\mathrm{H}^{0}\big( \mathrm{Cone}(E^{\ast} \rightarrow F^* ) \big) \cong \mathrm{HH}_{sg}^{0}(R). 
$$

Second, the composition of algebraic homomorphisms:
$$
\iota_R: R \xrightarrow{\pi_R} \mathrm{HH}_{sg}^{0}(R) 
\twoheadrightarrow \mathrm{HH}^{0}_{sg}(R)_{red}
$$
gives the canonical $R$-module structure on 
$\mathrm{HH}^{0}_{sg}(R)_{red}$.

Obviously $\iota_R$ induces an algebraic injection
$R/I \hookrightarrow \mathrm{HH}^{0}_{sg}(R)_{red}$ when $\mathrm{Ker}(\iota_R) \cong I$ as ideals of $R$. Now, we view $\mathrm{HH}^{0}_{sg}(R)_{red}$ as a sheaf on $\mathrm{Spec}(R)$. To prove the proposition,
it suffices to show that the support scheme of
$\mathrm{HH}^{0}_{sg}(R)_{red}$ is $V(I) \subseteq \mathrm{Spec}(R)$ since $V(\mathrm{Ker}(\iota_R)) \cong \mathrm{Supp}\Big( \mathrm{HH}^{0}_{sg}(R)_{red} \Big)$.  

Thus, it is enough to show that
the diagonal support scheme $\mathrm{DSupp}(S_{dg}(R))$ 
is exactly $V(I)$. 
In fact, the support scheme $\mathrm{Supp}\Big( \mathrm{HH}^{0}_{sg}(R)_{red} \Big)$ is $\mathrm{Supp}\Big( \mathrm{HH}^{0}_{sg}(R) \Big)$. 
Moreover,  
$\mathrm{Supp}\Big( \mathrm{HH}^{0}_{sg}(R) \Big)$  is also $\mathrm{DSupp}(S_{dg}(R))$ by definitions, since
the $R$-module structure on $\bigoplus_{i} \mathrm{Hom}_{D_{sg}(R^e)}(R, R[i])$ is given by the composition of $\pi_R$ and
the natural injection $\mathrm{HH}^{0}_{sg}(R) \hookrightarrow \bigoplus_{i} \mathrm{Hom}_{D_{sg}(R^e)}(R, R[i])$.

Let $\mathbf{p}$ be an element in
$\mathrm{Spec}(R)$. 
We know 
\begin{align*}
\left( \mathrm{Hom}_{S_{dg}(R^e)}(R, R) \right)_{\mathbf{p}} 
& \cong \mathrm{Cone}(E^{\ast} \rightarrow F^* )_{\mathbf{p}} \\
& \cong \mathrm{Cone}
(E^{\ast}_{\mathbf{p}} \rightarrow F^{*}_{\mathbf{p}} 
) 
\end{align*}
in $D(k)$. On the one hand, 
\begin{align*}
F^{*}_{\mathbf{p}} & =  \mathbb{R}\mathrm{Hom}_{R^e}(\mathrm{Bar}(R), R)_{\mathbf{p}} \\
& \cong \mathbb{R}\mathrm{Hom}_{R^e}(\mathrm{Bar}(R), R_{\mathbf{p}}) \\ 
& \cong \mathbb{R}\mathrm{Hom}_{R_{\mathbf{p}}^e}(\mathrm{Bar}(R) \otimes_{R^e} R_{\mathbf{p}}^e, R_{\mathbf{p}}) \\
& \cong \mathbb{R}\mathrm{Hom}_{R_{\mathbf{p}}^e}(\mathrm{Bar}(R_{\mathbf{p}}), R_{\mathbf{p}})
\end{align*}
in $D(k)$, where $\mathrm{Bar}(R)\otimes_{R^e} R_{\mathbf{p}}^e$ is 
a projective $R_{\mathbf{p}}^e$-module resolution of $R_{\mathbf{p}}$. 
On the other hand, 
\begin{align*}
E^{*}_{\mathbf{p}} & = \left(\mathrm{Bar}(R) \otimes_{R^e} \mathbb{R}\mathrm{Hom}_{R^e}\big(\mathrm{Bar}(R), R^e \big)\right)_{\mathbf{p}} \\ 
& \cong \left(R \otimes_{R^e} \overline{\mathbb{R}\mathrm{Hom}_{R^e}\big(\mathrm{Bar}(R), R^e \big)}\right)_{\mathbf{p}} \\
& \cong R_{\mathbf{p}} \otimes_{R^e} \overline{\mathbb{R}\mathrm{Hom}_{R^e}\big(\mathrm{Bar}(R), R^e \big)} \\ 
& \cong R_{\mathbf{p}} \otimes_{R_{\mathbf{p}}^e} R_{\mathbf{p}}^e \otimes_{R^e} \overline{\mathbb{R}\mathrm{Hom}_{R^e}\big(\mathrm{Bar}(R), R^e \big)} \\ 
& \cong (R_{\mathbf{p}} \otimes_{R_{\mathbf{p}}^e} R_{\mathbf{p}}^e) \otimes^{\mathbb{L}}_{R^e} \mathbb{R}\mathrm{Hom}_{R^e}\big(\mathrm{Bar}(R), R^e \big) \\ 
& \cong (R_{\mathbf{p}} \otimes^{\mathbb{L}}_{R_{\mathbf{p}}^e} R_{\mathbf{p}}^e) \otimes^{\mathbb{L}}_{R^e} \mathbb{R}\mathrm{Hom}_{R^e}\big(\mathrm{Bar}(R), R^e \big) \\ 
& \cong \mathrm{Bar}(R_{\mathbf{p}}) \otimes^{\mathbb{L}}_{R_{\mathbf{p}}^e} R_{\mathbf{p}}^e \otimes^{\mathbb{L}}_{R^e} \mathbb{R}\mathrm{Hom}_{R^e}\big(\mathrm{Bar}(R), R^e \big) \\
& \cong \mathrm{Bar}(R_{\mathbf{p}}) \otimes^{\mathbb{L}}_{R_{\mathbf{p}}^e} \Big(R_{\mathbf{p}}^e \otimes_{R^e} \mathbb{R}\mathrm{Hom}_{R^e}\big(\mathrm{Bar}(R), R^e \big) \Big)\\
& \cong \mathrm{Bar}(R_{\mathbf{p}}) \otimes^{\mathbb{L}}_{R_{\mathbf{p}}^e}   
\mathbb{R}\mathrm{Hom}_{R^e}\big(\mathrm{Bar}(R), R_{\mathbf{p}}^e \big) \\ 
& \cong \mathrm{Bar}(R_{\mathbf{p}}) \otimes^{\mathbb{L}}_{R_{\mathbf{p}}^e}   
\mathbb{R}\mathrm{Hom}_{R_{\mathbf{p}}^e}\big(\mathrm{Bar}(R) \otimes_{R^e} R_{\mathbf{p}}^e, 
R_{\mathbf{p}}^e \big) \\
& \cong \mathrm{Bar}(R_{\mathbf{p}}) \otimes^{\mathbb{L}}_{R_{\mathbf{p}}^e}   
\mathbb{R}\mathrm{Hom}_{R_{\mathbf{p}}^e}\big(\mathrm{Bar}(R_{\mathbf{p}}), R_{\mathbf{p}}^e \big)
\end{align*}
in $D(k)$, where $\overline{\mathbb{R}\mathrm{Hom}_{R^e}\big(\mathrm{Bar}(R), R^e \big)}$ is a flat  
resolution of $\mathbb{R}\mathrm{Hom}_{R^e}\big(\mathrm{Bar}(R), R^e \big)$ 
over $R^e$. By the result, we obtain 
$$
 \mathrm{Cone}(E^{\ast}_{\mathbf{p}} \rightarrow F^{*}_{\mathbf{p}} ) 
 \cong \mathrm{Hom}_{S_{dg}(R_{\mathbf{p}}^e)}(R_{\mathbf{p}}, R_{\mathbf{p}})
$$
in $D(R_{\mathbf{p}})$, if $R$ is replaced by $R_{\mathbf{p}}$.
Thus, we get 
$$
\left( \mathrm{Hom}_{S_{dg}(R^e)}(R, R) \right)_{\mathbf{p}} 
\cong \mathrm{Hom}_{S_{dg}(R_{\mathbf{p}}^e)}(R_{\mathbf{p}}, R_{\mathbf{p}})
$$
as algebras. 

Now assume $\mathbf{p}$ is a non-singular point, then $R_{\mathbf{p}}$ is a homologically smooth algebra, which implies that $\mathrm{Hom}_{S_{dg}(R_{\mathbf{p}}^e)}(R_{\mathbf{p}}, R_{\mathbf{p}})$ 
is trivial in $D(R_\mathbf{p})$. 
Hence $\left( \mathrm{Hom}_{S_{dg}(R^e)}(R, R) \right)_{\mathbf{p}}$ 
is also trivial in $D(R_\mathbf{p})$. 
Thus, ${\mathbf{p}}$ is not contained in
the diagonal support scheme $\mathrm{DSupp}(S_{dg}(R))$, 
which implies
$$
\mathrm{DSupp}(S_{dg}(R)) \subseteq V(I) := \mathrm{Spec}(R/I). 
$$

Next, let $\mathbf{q}$ be a singular point in 
$\mathrm{Spec}(R)$, i.e., $\mathbf{q} \in V(I)$. 
After similar arguments as above, we know
$\left( \mathrm{Hom}_{S_{dg}(R^e)}(R, R) \right)_{\mathbf{q}} 
\cong \mathrm{Hom}_{S_{dg}(R_{\mathbf{q}}^e)}(R_{\mathbf{q}}, R_{\mathbf{q}})$.

Assume that $\mathrm{Hom}_{S_{dg}(R_{\mathbf{q}}^e)}(R_{\mathbf{q}}, 
R_{\mathbf{q}})$ is trivial in $D(R_\mathbf{q})$. 
It implies that $\mathrm{HH}^{0}_{sg}(R_{\mathbf{q}}) \cong \mathrm{Hom}_{D_{sg}(R_{\mathbf{q}}^e)}(R_{\mathbf{q}}, 
R_{\mathbf{q}})$ is trivial algebra.
Now, consider the surjective linear map: 
$$
\mathrm{Hom}_{D_{sg}(R_{\mathbf{q}}^e)}(R_{\mathbf{q}} R_{\mathbf{q}}) \otimes_{k} \mathrm{Hom}_{D_{sg}(R_{\mathbf{q}}^e)}(F, R_{\mathbf{q}})  \twoheadrightarrow \mathrm{Hom}_{D_{sg}(R_{\mathbf{q}}^e)}(F, R_{\mathbf{q}}),
$$
which arises from the associativity of composition in the category $D_{sg}(R_{\mathbf{q}}^e)$,  
for any object $F$ in $D_{sg}(R_{\mathbf{q}}^e)$. 
Since $\mathrm{Hom}_{D_{sg}(R_{\mathbf{q}}^e)}(R_{\mathbf{q}}, 
R_{\mathbf{q}})$ is is trivial, the left-hand side vanishes. Thus, we have 
$$
\mathrm{Hom}_{D_{sg}(R_{\mathbf{q}}^e)}(F, R_{\mathbf{q}}) \cong 0
$$ 
for any object $F$ in $D_{sg}(R_{\mathbf{q}}^e)$. 
By the Yoneda lemma, $R_{\mathbf{q}} \cong 0$ in $D_{sg}(R_{\mathbf{q}}^e)$, which implies that $R_{\mathbf{q}}$ is homologically smooth. 
Thus we obtain that 
$\mathbf{q}$ is a non-singular point in 
$\mathrm{Spec}(R)$. It is a contradiction to 
the fact that 
$\mathbf{q}$ is singular in 
$\mathrm{Spec}(R)$. 
Hence the assumption that 
$$
\mathrm{Hom}_{S_{dg}(R_{\mathbf{q}}^e)}\big(\mathrm{Bar}( R_{\mathbf{q}}), 
\mathrm{Bar}( R_{\mathbf{q}})\big) 
\cong \mathrm{Hom}_{S_{dg}(R_{\mathbf{q}}^e)}( R_{\mathbf{q}}, R_{\mathbf{q}})
$$ 
is trivial in $D(R_\mathbf{q})$ does
not hold. 
Thus
$\mathrm{Hom}_{S_{dg}(R_{\mathbf{q}}^e)}(R_{\mathbf{q}},  R_{\mathbf{q}})$ 
is nontrivial in $D(R_\mathbf{q})$. 
From the isomorphism 
$$
\left( \mathrm{Hom}_{S_{dg}(R^e)}(R, R) \right)_{\mathbf{q}} 
\cong \mathrm{Hom}_{S_{dg}(R_{\mathbf{q}}^e)}(R_{\mathbf{q}}, R_{\mathbf{q}}),
$$
as algebras in the above argument, we know that $\mathbf{q}$ is contained in 
$\mathrm{DSupp}(S_{dg}(R))$, which suggests that   
$$
V(I) \subseteq \mathrm{DSupp}(S_{dg}(R)).
$$
Then we obtain $V(I) = \mathrm{DSupp}(S_{dg}(R))$,
which means that
the diagonal support scheme of 
$\mathrm{Hom}_{S_{dg}(R^e)}(R, R)$  
is $V(I)$.	
Thus, we complete the proof. 	
\end{proof}

Via the algebra homomorphism $\iota_{I}$, $\mathrm{HH}^{0}_{sg}(R)_{red}$ can be viewed as an $R/I$-module. 

\begin{lemma}\label{Fini}
With the above setting, $\mathrm{HH}^{0}_{sg}(R)_{red}$ is a finitely 
generated $R/I$-module.    
\end{lemma}

\begin{proof} 	
First, by Proposition \ref{MainPro}, we have 
$\mathrm{HH}_{sg}^{0}(R)= \mathrm{H}^{0}\big( \mathrm{Cone}(E^{\ast} \rightarrow F^* ) \big)$.

Second, $R^e$ is also a Gorenstein 
Noetherian algebra, since $R$ is a Gorenstein Noetherian algebra.
Hence, $R^e$ admits a bounded injective $R^e$-module resolution,
denoted by $J^*$.
There exist two quasi-isomorphisms 
$$
\mathbb{R}\mathrm{Hom}_{R^e}(\mathrm{Bar}(R), J^*) 
\rightarrow \mathbb{R}\mathrm{Hom}_{R^e}(\mathrm{Bar}(R), R^e)
$$ 
via the injective resolution $R^e \hookrightarrow  J^*$ 
and
$$
\mathbb{R}\mathrm{Hom}_{R^e}(R, J^*) \rightarrow \mathbb{R}\mathrm{Hom}_{R^e}(\mathrm{Bar}(R), J^*)
$$
via the Bar resolution $\mathrm{Bar}(R) \twoheadrightarrow R$. 
Due to the boundedness of $J^{\ast}$, there are only finite parts of the complex 
$\mathbb{R}\mathrm{Hom}_{R^e}(R, J^*)$ have nontrivial cohomologies. 

Meanwhile, $\forall i \in \mathbb{Z}$, $\mathrm{Hom}_{K(R^e)}(P^{*}_R, R^{e}[i]) \cong \mathrm{Hom}_{D(R^e)}(R, R^{e}[i])$  for
any finitely generated projective resolution $P^{*}_R$ of $R$ as $R^e$-modules, where
$K(R^e)$ is the homotopy category of complexes of $R^e$-modules. Hence we have that 
for any $i \in \mathbb{Z}$,  
$\mathrm{Hom}_{D(R^e)}(R, R^{e}[i])$ is a finitely generated $R^e$-module.

By \cite[Proposition 3.5]{DH}, we find that
$\mathbb{R}\mathrm{Hom}_{R^e}(\mathrm{Bar}(R), J^*)$ is quasi-isomorphic to a complex of finitely generated bounded $R^e$-modules.
Thus, 
$\mathbb{R}\mathrm{Hom}_{R^e}(\mathrm{Bar}(R), J^*)$ admits a finitely generated projective $R^e$-module resolution, denoted by $P^{\ast}_{J^*}$.  
It suggests that there are quasi-isomorphisms
\begin{align*}
& \mathrm{Bar}(R) \otimes_{R^e} P^{\ast}_{J^*} \xrightarrow{\sim} 
\mathrm{Bar}(R) \otimes_{R^e} \mathbb{R}\mathrm{Hom}_{R^e}(R, J^*) 
 \xrightarrow{\sim} \mathrm{Bar}(R) \otimes_{R^e} 
 \mathbb{R}\mathrm{Hom}_{R^e}(\mathrm{Bar}(R), J^*) \\
 & \xrightarrow{\sim} 
\mathrm{Bar}(R) \otimes_{R^e} \mathbb{R}\mathrm{Hom}_{R^e}(\mathrm{Bar}(R), R^e) 
\xrightarrow{\sim}  E^{*}.
\end{align*}
Moreover, there is a quasi-isomorphism  
$$
\mathrm{Bar}(R) \otimes_{R^e} P^{\ast}_{J^*} \xrightarrow{\sim} R \otimes_{R^e} P^{\ast}_{J^*}.  
$$ 
Thus, we have  
 $$
 R \otimes_{R^e} P^{\ast}_{J^*} \xrightarrow{\sim} E^{*}
 $$
 in $D(k)$. 
Meanwhile, we know that $P^{j}_{J^*}$ is a finitely generated projective $R^e$-module for any $j$.   
Thus, $R \otimes_{R^e} P^{j}_{J^*} $ is a finitely generated $R$-module for any $j$.    
Since $R \otimes_{R^e} P^{\ast}_{J^*} $ is a finitely generated $R$-module complex, 
the cohomologies of $E^{\ast}$ are all finitely generated $R$-modules. 

From the distinguished triangle  	
$$
F^{\ast} \rightarrow \mathrm{Cone}(E^{\ast} \rightarrow F^* ) \rightarrow E^{\ast}[1], 
$$	
we have the long exact sequence 
$$
\cdots \rightarrow \mathrm{H}^{0}(E^{\ast}) \rightarrow \mathrm{H}^{0}(F^{\ast})
 \xrightarrow{\pi_R} \mathrm{H}^{0}\big( \mathrm{Cone}(E^{\ast} \rightarrow F^* ) \big) 
 \rightarrow \mathrm{H}^{1}(E^{\ast}) \rightarrow \cdots,
$$
where $\mathrm{H}^{0}(F^{\ast}) \cong \mathrm{HH}^{0}(R) \cong R$ and $\mathrm{H}^{0}\big( \mathrm{Cone}(E^{\ast} \rightarrow F^* ) \big) 
\cong \mathrm{HH}_{sg}^{0}(R)$. 
Since both $\mathrm{H}^{0}(F^{\ast})$ and $\mathrm{H}^{1}(E^{\ast})$ 
are finitely generated $R$-modules, we find
that $\mathrm{HH}_{sg}^{0}(R)$ is also a finitely generated
$R$-module from the long exact sequence. Thus,  
 $\mathrm{HH}_{sg}^{0}(R)_{red}$ is a finitely generated
$R$-module 
 
Since $\iota_{R}$ factors through the algebraic homomorphism 
$\iota_{I}: R/I \hookrightarrow \mathrm{HH}_{sg}^{0}(R)_{red}$   
 (see Proposition \ref{Injec}), 
we obtain that $\mathrm{HH}_{sg}^{0}(R)_{red}$ is a finitely generated $R/I$-module. 
Eventually, we complete the proof. 
\end{proof}

\section{Proof of the main result}\label{TR}

In this section, we prove Theorem \ref{Th1}. 

\begin{proof}[Proof of Theorem \ref{Th1}]
Keep the settings in Proposition \ref{Injec} and Lemma \ref{Fini}. 
Since the algebraic homomorphism $\iota_I$ gives the $R/I$-module 
structure on $\mathrm{HH}_{sg}^{0}(R)_{red}$,  
the image of $\iota_I$ is contained in the center 
$Z\big(\mathrm{HH}_{sg}^{0}(R)_{red} \big)$ of $\mathrm{HH}_{sg}^{0}(R)_{red}$. 
Note that $\mathrm{HH}_{sg}^\bullet(R)$ is a Gerstenhaber algebra,
hence $\mathrm{HH}_{sg}^0(R)$ is a commutative subalgebra of
 $\mathrm{HH}_{sg}^\bullet(R)$ endowed with the cup product \cite{W0}. 
Therefore, $Z(\mathrm{HH}_{sg}^0(R)_{red})=\mathrm{HH}_{sg}^0(R)_{red}$.
Now, we have a homomorphism between commutative rings
$$
\iota_I: R/I \hookrightarrow \mathrm{HH}_{sg}^{0}(R)_{red}. 
$$
By Lemma \ref{Fini}, it is clear that
$\iota_I$ is a finite morphism between commutative rings, i.e.,
$\mathrm{HH}_{sg}^{0}(R)_{red}$ is a finitely generated $R/I$-module. 
Hence, $\mathrm{HH}_{sg}^{0}(R)_{red}$ is also a Noetherian ring. Moreover, 
by Proposition \ref{Injec}, $\iota_I$ is also a finite injection.  
Combining above arguments, 
it follows that there exists a surjective morphism between schemes:  
$$
(\iota_I)_{\natural}: \mathrm{Spec}\left(\mathrm{HH}_{sg}^{0}(R)_{red} \right) 
\twoheadrightarrow \mathrm{Spec}(R/I)
$$
given by $\iota_I$. Furthermore, $(\iota_I)_{\natural}$ is also a finite morphism 
between  schemes by the finiteness property of $\iota_I$. 
Then we complete the proof.
\end{proof}

From the above theorem, we obtain the following.

\begin{corollary}\label{MainLem}
Let $R_1$ and $R_2$	
be two commutative Gorenstein Noetherian rings. 
Suppose that there is a DG singular equivalence  	
$$
F: S_{dg}(R_1) \xrightarrow{\sim}
S_{dg}(R_2).
$$	
Then the two singular loci in the affine schemes of these two rings have the same dimension. 
\end{corollary} 

\begin{proof}
Again keep the settings in Proposition \ref{Injec} and Lemma \ref{Fini}. 
Since $(\iota_I)_{\natural}$ is a finite morphism between schemes, 
$(\iota_I)_{\natural}$ has finite fibers, i.e., any one of its fibers has finite points. 
It implies that  
$$
\mathrm{dim}\big( \mathrm{Spec}(R/I) \big) = \mathrm{dim}
\left( \mathrm{Spec}\left( \mathrm{HH}_{sg}^{0}(R)_{red} \right) \right) 
$$
since $(\iota_I)_{\natural}$ is surjective. 
Thus the dimension of singular 
locous of $\mathrm{Spec}(R)$ is equal to the Krull dimension of $ \mathrm{HH}_{sg}^{0}(R)_{red} $.  

In the meantime, there is a quasi-equivalence  	
$$
F: S_{dg}(R_1) \xrightarrow{\sim} S_{dg}(R_2). 
$$  
Thus by Theorem \ref{Th2}, their Hochschild cohomologies are isomorphic. 
It implies that 
$$
\mathrm{HH}_{sg}^{0}(R_1)_{red} \cong \mathrm{HH}_{sg}^{0}(R_2)_{red} 
$$
as algebras. Therefore, the two singular loci in these two schemes
of $R_1$ and $R_2$ respectively,  have the same dimension. 
\end{proof}

By Proposition \ref{Morita}, 
any singularity equivalence of Morita type gives a DG singularity equivalence,
and thus we get the following. 
\begin{corollary}\label{MainLem1}
Let $R_1$ and $R_2$	
be two commutative Gorenstein Noetherian rings. 
Suppose that there is a singular equivalence of Morita type
$$
\Phi: D_{sg}(R_1) \xrightarrow{(-)\otimes^{\mathbb{L}}_{R_1} M} 
D_{sg}(R_2),
$$	
for some $R_{2} \otimes R^{op}_1$-module $M$. Then the two singular loci in the affine schemes 
of these two rings have the same dimension. 	
\end{corollary}

\begin{ack}
We would like to thank Xiaojun Chen and Huijun Fan 
for their encouragement, support and suggestions. 
Moreover, we thank Martin Kalck for 
pointing out an error
in the previous version of this paper. This paper is supported by the National Natural Science Foundation of China (Grant No. 12401050).
\end{ack}

\end{document}